\documentclass[11pt,fleqn]{amsart}
\usepackage{amssymb}
\usepackage{amsmath}
\usepackage{amsthm}
\usepackage[all,cmtip]{xy}
\allowdisplaybreaks
\usepackage{amscd}
\usepackage{graphicx}

\usepackage{ifpdf}
\ifpdf
  \usepackage[colorlinks=true,linkcolor=blue,final,backref=page,hyperindex,citecolor=red]{hyperref}
\else
  \usepackage[colorlinks,final,backref=page,hyperindex,hypertex]{hyperref}
\fi

\usepackage{geometry}
\usepackage{pxfonts}

\usepackage[shortlabels]{enumitem}
\usepackage{ifpdf}

\def\vbar{\mathchoice{\vrule height6.3ptdepth-.5ptwidth.8pt\kern- .8pt}
{\vrule height6.3ptdepth-.5ptwidth.8pt\kern-.8pt} {\vrule
height4.1ptdepth-.35ptwidth.6pt\kern-.6pt} {\vrule
height3.1ptdepth-.25ptwidth.5pt\kern-.5pt}}

\usepackage{tikz}
\usepackage[active]{srcltx}

\makeatletter

\newtheorem{thm}{Theorem}[section]
\newtheorem{lem}[thm]{Lemma}
\newtheorem{cor}[thm]{Corollary}
\newtheorem{pro}[thm]{Proposition}
\newtheorem{ex}[thm]{Example}

\newtheorem{defi}[thm]{Definition}

\newcommand{\be }{\begin{equation}}
\newcommand{\ee }{\end{equation}}

\newcommand{\huaL}{\mathcal{L}}

\newcommand{\Id}{\rm{Id}}

\newcommand{\br}[1]{   [ \cdot,    \cdot  ]   }

\newcommand{\gl}{\mathfrak {gl}}

\newcommand{\ad}{\mathrm{ad}}

\begin{document}

\title{On Ternary $F$-manifold Algebras and their Representations}

\author{ A.  Ben Hassine }
\address{Department of Mathematics, Faculty of Science and Arts at
Belqarn, University of Bisha, Kingdom of Saudi Arabia \&
Faculty of Sciences, University of Sfax, Tunisia }
\email{   Benhassine@ub.edu.sa }

\author{ T.  Chtioui }
\address{  Faculty of Sciences, University of Sfax, Tunisia }
\email{ chtioui.taoufik@yahoo.fr  }

\author{M.  Elhamdadi}
\address{Department of Mathematics,
	University of South Florida, Tampa, FL 33620, U.S.A.}
\email{emohamed@math.usf.edu}

\author{ S.  Mabrouk  }
\address{ Faculty of Sciences, University of Gafsa,   BP
2100, Gafsa, Tunisia }
\email{ sami.mabrouk@fsgf.u-gafsa.tn, 
 mabrouksami00@yahoo.fr  }

\maketitle

\date{}

\begin{abstract}
We introduce a notion of ternary $F$-manifold algebras which is a generalization of $F$-manifold algebras.  We study representation theory of ternary $F$-manifold algebras.
In particular, we introduce a notion of dual representation which requires additional conditions 
similar to the binary case. Then, we establish a notion of a coherence ternary $F$-manifold algebra. Moreover, we investigate the construction of ternary $F$-manifold algebras using $F$-manifold algebras. Furthermore, we introduce and investigate a notion of a relative Rota-Baxter operator with respect to a representation and use it to construct ternary pre-$F$-manifold algebras.
\end{abstract}

\noindent\textbf{Keywords:} Ternary $F$-manifold algebra, representation, relative Rota-Baxter operators, ternary pre-$F$-manifold algebras.

\noindent{\textbf{MSC(2020):}}17B10, 17B56, 17A42

\tableofcontents

\allowdisplaybreaks


\section{Introduction}\label{sec:intr}

Ternary Lie algebras and more generally $n$-ary Lie algebras are natural generalizations of Lie algebras. They were introduced and studied by Filippov in \cite{Filippov}.
This type of algebras appeared also in the algebraic formulation of Nambu Mechanics \cite{Nambu}, generalizing Hamiltonian mechanics by considering two hamiltonians (see \cite{Gautheron,Takhtajan}). Moreover, $3$-Lie algebras appeared in string theory and $M$-theory.  Basu and Harvey suggested to replace the Lie algebra appearing in the Nahm equation by a $3$-Lie algebra for the lifted Nahm equations (\cite{Basu&Harvey}. Furthermore,  Bagger-Lambert managed to construct, using a ternary bracket, an $N=2$ supersymmetric version of the world volume
theory of the $M$-theory membrane, in the context of Bagger-Lambert-Gustavsson model of multiple $M2$-branes  (for more details see \cite{Bagger&Lambert2008,Bagger&Lambert2009,Gomis,Ho&Hou&Matsuo,Papadopoulos}.
The notion of $n$-Lie algebra was studied in 
recent years in several algebraic contexts.  In \cite{Bai&Wang,Bai&Song},  the authors gave 
construction, realization and classifications of $3$-Lie algebras and $n$-Lie algebras. Representation theory of $n$-Lie algebras was first introduced by Kasymov in \cite{repKasymov} and 
cohomology theory of Fillippov algebras was 
studied in \cite{Izquierdo}. The adjoint representation is defined by the ternary bracket in which two elements are fixed.
The notion of   Rota-Baxter operators  on associative algebras was introduced in 1960 by G.Baxter \cite{Baxter1960} in his study of fluctuation theory in probability. Recently, there has been many applications, including in Connes-Kreimer's algebraic approach to renormalization in perturbative quantum field theory \cite{Connes}. In the Lie algebra context, a Rota-Baxter operator of weight zero was introduced independently in the 1980s as the operator form of the classical Yang Baxter equation, whereas the classical Yang-Baxter equation plays important roles in many
areas of mathematics and mathematical physics such as quantum groups and integrable systems \cite{Chari,Semonov}. Rota-Baxter operators on super-type algebras were studied in \cite{Abdaoui&Mabrouk&Makhlouf}, which build relationships between associative superalgebras, Lie superalgebras, $L$-dendriform superalgebras and pre-Lie superalgebras. Recently relative Rota-Baxter operators on Leibniz algebras were studied in \cite{Sheng&Tang}. The notion of relative Rota-Baxter operators  on $3$-Lie algebras with respect to a representation thanks to the solutions of $3$-Lie classical Yang-Baxter equation \cite{Bai&Guo&Sheng}. In particular, Rota-Baxter operators on $3$-Lie algebras, introduced in \cite{Bai&Guo&Li&Wu}, are relative Rota-Baxter operators on a $3$-Lie algebra with respect to the adjoint representation.\\
 On the other  hand, in order to give a geometrical expression of the WDV equations, Dubrov introduced the notion of Frobenius manifolds \cite{DS11}. The concept of $F$-manifolds was introduced by Hertling and Manin, as a relaxation of the conditions of Frobenius manifolds  \cite{HerMa}. $F$-manifolds appear in many 
 areas of mathematics such as singularity theory \cite{Her02}, quantum $K$-theory \cite{LYP}, integrable systems \cite{DS04,DS11,LPR11},  operads \cite{Merku} and so on. Inspired by the investigation of algebraic structures of $F$-manifolds, the notion of  $F$-manifold 
algebras 
was given by Dotsenko in \cite{Dot}. Recently,
the notion  of representations of $F$-manifold algebras
 which gives rise to pre-$F$-manifold algebras
was introduced in \cite{Liu&Sheng&Bai}. Later, the notion of $F$-manifold color algebras and Hom-$F$-manifolds algebras and some constructions were given in (\cite{Hassine&Chtioui&Maalaoui&Mabrouk,Ming&Chen&Li}).


A natural question of introducing a notion of ternary $F$-manifold algebras then arises.  This is the goal of this article.
We introduce a notion of ternary $F$-manifold algebras by first, defining a structure which can be induced by 
an $F$-manifold algebra exactly as for $3$-Lie algebras.  Second, the structure must generalize ternary Nambu-Poisson algebras.\\
The paper is organized as follows. In Section \ref{sec:Preliminaries}, 
we review some basics of 
$3$-Lie algebras, $3$-pre-Lie algebras,  ternary Nambu-Poisson algebras and $F$-manifold algebras. In Section \ref{sec:ternary F-algebras and deformations}, we introduce a notion of \emph{ternary $F$-manifold} algebras. 
We also give a construction of ternary $F$-manifold algebras with the help of a given $F$-manifold algebra equipped with a trace function. 
Section \ref{representationternarymanifold} is dedicated to 
representations and dual representations of ternary $F$-manifold algebras.  In Section~\ref{preF} we introduce a notion of ternary pre-$F$-manifold algebras and relative Rota-Baxter operators on a ternary $F$-manifold algebra. We show that on one hand, a relative Rota-Baxter operator on a ternary $F$-manifold algebra gives rise to a ternary pre-$F$-manifold algebra. On the other hand, a ternary pre-$F$-manifold algebra is naturally built via symplectic structures. 
 
 In this paper, all the vector spaces are over an algebraically closed field $\mathbb K$ of characteristic $0$.
\vspace{2mm}

\section{Preliminaries}\label{sec:Preliminaries}~~
An  associative  algebra is a pair $(A,\cdot)$, where $A$ is a vector space and  $\cdot:A\otimes A\longrightarrow A$ is  a   linear map satisfying that for all $x,y,z\in A$, the  associator
$\mathfrak{as} (x,y,z)=(x\cdot y)\cdot z-x\cdot(y\cdot z)=0$,
i.e.
$$(x\cdot  y)\cdot z=x\cdot (y\cdot  z).$$
Furthermore, if  $x\cdot y=y\cdot x$ for all $x,y\in A$, then $(A,\cdot)$ is called a commutative associative algebra.

    A representation of commutative associative  algebra $(A,\cdot)$ on
  a vector space $V$ is a linear map
  $\mu :A\longrightarrow \mathfrak{gl}(V)$, such that for any
  $x,y\in A$, the following equalities are satisfied:
  \begin{eqnarray}
 \label{representation-ass2}
    \mu(x\cdot y)
    &=&\mu (x) \mu(y),\forall x,y\in A.
  \end{eqnarray}
If there exists an element  $1_A\in A$ satisfying $x\cdot  1_A= 1_A \cdot x=x$, then $A$ is called unital with unit $1_A$.\\

Let $(V;\mu)$ be a representation of a commutative associative algebra $(A,\cdot)$.
 In the sequel, define $\mu^*:A\longrightarrow \gl(V^*)$  by
$$
 \langle \mu^*(x)\xi,v\rangle=-\langle \xi,\mu(x)v\rangle,\quad \forall ~ x\in A,\xi\in V^*,v\in V.
$$

\begin{lem}
Under the above notation,   $(V^*,\mu^*)$ is a representation
 of $(A,\cdot)$ which is called the  dual representation of  $(V;\mu)$.
\end{lem}

  A Zinbiel algebra  is a pair $(A,\diamond)$, where
$A$ is a vector space,   $\diamond:A\otimes A\longrightarrow A$ is
a binary  multiplication   such that for all $x,y,z\in A$,
\begin{equation}
 x\diamond(y\diamond z)=(y\diamond x)\diamond z+(x\diamond y)\diamond z.
\end{equation}

 \begin{lem}\label{lem:den-ass}
Let $(A,\diamond)$ be a Zinbiel algebra. Then $(A,\cdot)$ is a commutative associative algebra, where $x\cdot y=x\diamond y+y\diamond x$. Moreover, for $x\in A$, define $L_\diamond(x):A\longrightarrow\gl(A)$ by
\begin{equation}\label{eq:dendriform-rep}
L_\diamond( x)(y)=x\diamond y,\quad\forall~y\in A.
\end{equation}
Then $(A;L_\diamond)$ is a representation of the commutative associative algebra $(A,\cdot)$.
\end{lem}

\begin{defi}
   A $3$-Lie algebra  consists  of a vector space $A$ equipped with a skew-symmetric linear map called $3$-Lie bracket $[\cdot,\cdot,\cdot]:
\otimes^3 A\rightarrow A$ such that the following Fundamental Identity holds (for $x_i\in A, 1\leq i\leq 5$)
\begin{equation}\label{eq:de1}
[x_1,x_2,[x_3,x_4,x_5]]=[[x_1,x_2,x_3],x_4,x_5]+[x_3,[x_1,x_2,x_4],x_5]+[x_3,x_4,[x_1,x_2,x_5]]
\end{equation}
\end{defi}
In other words, for $x_1, x_2\in A$, the operator
\begin{align}\label{eq:adjoint}
ad_{x_1,x_2}:A\to A, \quad ad_{x_1,x_2}x:=[x_1,x_2,x], \quad \forall x\in A,
\end{align}
is a derivation in the sense that
$$ ad_{x_1,x_2}[x_3,x_4,x_5]=[ad_{x_1,x_2}x_3,x_4,x_5] +[x_3,ad_{x_1,x_2}x_4,x_5]+[x_3,x_4,ad_{x_1,x_2}x_5], \forall x_3, x_4, x_5\in A.$$
A morphism between 3-Lie algebras is a linear map that preserves the 3-Lie brackets.
\begin{ex}\cite{Filippov} Consider $4$-dimensional $3$-Lie algebra $A$ generated by $(e_1,e_2,e_3,e_4)$  with  the following multiplication
\begin{align*}
          [e_1,e_2,e_3]= e_4, \
[e_1,e_2,e_4]= e_3,\
[e_1,e_3,e_4]= e_2, \
     [ e_2,e_3,e_4]= e_1.
\end{align*}
\end{ex}

The notion of a representation of an $n$-Lie algebra was introduced in \cite{repKasymov}. See also \cite{rep}.
\begin{defi}\label{defi:rep}
 A representation of a 3-Lie algebra $A$ on a vector space $V$ is a skew-symmetric linear map $\rho: \otimes^2A\rightarrow gl(V)$ satisfying
\begin{enumerate}
\item [\rm(i)] $\rho (x_1,x_2)\rho(x_3,x_4)-\rho(x_3,x_4)\rho(x_1,x_2)=
\rho([x_1,x_2,x_3],x_4)-\rho([x_1,x_2,x_4],x_3)$,
\item [\rm (ii)]$\rho ([x_1,x_2,x_3],x_4)=\rho(x_1,x_2)\rho(x_3,x_4)+\rho(x_2,x_3)\rho(x_1,x_4)+\rho(x_3,x_1)\rho(x_2,x_4)$,
\end{enumerate}
for $x_i\in A, 1\leq i\leq 4$.
\end{defi}
  Let $(V, \rho)$ be a representation of a $3$-Lie algebra $A$. Define $\rho^\ast :\otimes^2A \to \mathfrak{gl}(V^\ast)$ by
\begin{equation}
    \langle\rho^\ast(x_1, x_2)\xi, v\rangle = -\langle \xi, \rho(x_1, x_2)v\rangle, \forall \xi \in  V^\ast, x_1, x_2 \in A, v \in V.
\end{equation}
\begin{pro}
  With the above notations, $(V^\ast, \rho^\ast)$ is a representation of $A$, called the dual representation.
\end{pro}
\begin{ex}
 Let $A$ be a $3$-Lie algebra. The linear map $ad :\otimes^2A\to \mathfrak{gl}(A)$ with $x_1, x_2 \to
ad_{x_1,x_2}$ for any $x_1, x_2 \in A$ defines a representation $(A, ad)$ which is called the adjoint
representation of $A$, where $ad_{x_1,x_2}$ is given by Eq. \eqref{eq:adjoint}. The dual representation
$(A^\ast, ad^\ast)$ of the adjoint representation $(A, ad)$ of a $3$-Lie algebra $A$ is called the coadjoint
representation.
\end{ex}

\begin{defi}
  A $3$-pre-Lie algebra is a pair $(A,\{\cdot,\cdot,\cdot\})$ consisting of a
a vector space $A$ and  a linear map $\{\cdot,\cdot,\cdot\}:A\otimes A\otimes A\rightarrow A$ such that the following identities hold:
\begin{eqnarray}
\{x,y,z\}&=&-\{y,x,z\},\label{3-pre-Lie 0}\\
\nonumber\{x_1,x_2,\{x_3,x_4,x_5\}\}&=&\{[x_1,x_2,x_3]^C,x_4,x_5\}+\{x_3,[x_1,x_2,x_4]^C,x_5\}\\
&&+\{x_3,x_4,\{x_1,x_2,x_5\}\},\label{3-pre-Lie 1}\\
\nonumber\{ [x_1,x_2,x_3]^C,x_4, x_5\}&=&\{x_1,x_2,\{ x_3,x_4, x_5\}\}+\{x_2,x_3,\{ x_1,x_4,x_5\}\}\\
&&+\{x_3,x_1,\{ x_2,x_4, x_5\}\},\label{3-pre-Lie 2}
\end{eqnarray}
 where $x,y,z, x_i\in A, 1\leq i\leq 5$ and $[\cdot,\cdot,\cdot]^C$ is defined by
\begin{equation}
[x,y,z]^C=\{x,y,z\}+\{y,z,x\}+\{z,x,y\},\quad \forall  x,y,z\in A.\label{eq:3cc}
\end{equation}
\end{defi}
\begin{pro}\label{3preLieTo3Lie}
Let $(A,\{\cdot,\cdot,\cdot\})$ be a $3$-pre-Lie algebra. Then the induced $3$-commutator given by Eq.~\eqref{eq:3cc} defines
a $3$-Lie algebra.
\end{pro}

Let $(A, \{\cdot,\cdot,\cdot\})$ be a $3$-pre-Lie algebra. Define a skew-symmetric linear map $\mathbb{L}: \otimes^2 A \to  gl(A)$
by
\begin{align}
\mathbb{L}(x, y)z = \{x, y, z\},\quad \forall x, y, z \in A. \label{ed 3 pre lie}
\end{align}
By the definitions of a $3$-pre-Lie algebra and a representation of a $3$-Lie algebra, we have

\begin{lem}\label{adjoint 3prelie}
With the above notations, $(A, \mathbb{L})$ is a representation of the $3$-Lie algebra
$(A, [\cdot,\cdot,\cdot]^C )$. 
\end{lem}

\begin{defi}\cite{Makhlouf&Amri}
  A  ternary Nambu-Poisson algebra is a triple $(A,\cdot,[\cdot,\cdot,\cdot])$ consisting of  a $\mathbb{K}$-vector space $A$,  two linear maps $\cdot:A\otimes A\rightarrow A$ and $[\cdot,\cdot,\cdot]: A\otimes A\otimes A\rightarrow A$ such that
\begin{enumerate}
  \item $(A,\cdot)$ is a   commutative associative algebra,
  \item $(A,[\cdot,\cdot,\cdot])$ is a $3$-Lie algebra,
 \item the following Leibniz rule
 \begin{align*}
[x_{1},x_{2}, x_{3}\cdot x_{4}]= x_{3}\cdot[x_{1},x_{2},x_{4}]+[x_{1},x_{2},x_{3}]\cdot x_{4}
\end{align*}
\end{enumerate}
holds for  all $x_{1},x_{2},x_{3}\in A$.
\end{defi}

Now, we recall the notion of $F$-manifold algebra given in \cite{Liu&Sheng&Bai}.

\begin{defi}A   $F$-manifold algebra is a tuple $(A,\cdot,[\cdot,\cdot])$, where $(A,\cdot)$ is a commutative associative algebra and $(A,[\cdot,\cdot])$ is a Lie algebra, such that for all $x,y,z,w\in A$, the Hertling-Manin relation holds
:
\begin{equation}\label{eq:HM relation}
\mathcal L(x\cdot y,z,w)=x\cdot \mathcal L(y,z,w)+y\cdot \mathcal L(x,z,w),
\end{equation}
where $ \mathcal L(x,y,z)$ is the Leibnizator define by
\begin{equation}
 \mathcal L(x,y,z)=[x,y\cdot z]-[x,y]\cdot z-y\cdot [x,z].
\end{equation}

\end{defi}

Let $(A,\cdot ,[\cdot,\cdot])$ be a  $F$-manifold algebra, $(V;\rho )$ be a
  representation of the  Lie algebra $(A,[\cdot,\cdot] )$ and $(V;\mu)$ be a representation of the commutative associative algebra $(A,\cdot)$. Define the three linear maps $\mathcal{L}_1,\mathcal{L}_2 :A\otimes A\otimes V\to V$
  given by
  \begin{eqnarray}
  \label{eq:repH 1}\mathcal{L}_1(x,y,u)&=&\rho(x)\mu(y)(u)-\mu(y) \rho(x)(u)-\mu([x,y])(u),\\
  \label{eq:repH 2}\mathcal{L}_2(x,y,u)&=&\mu(x)\rho(y)(u)+\mu(y) \rho(x)(u)-\rho(x\cdot y)(u),
   \end{eqnarray}
  for all $x,y\in A$ and $u\in V$. 
\begin{defi}\label{def-rep-manifo} With the above notations, the tuple $(V;\rho,\mu)$ is a representation of $A$  if the following conditions hold:
   \begin{eqnarray}
     \label{eq:rep 1}\mathcal{L}_1(x\cdot  y,z,u)&=\mu(x) \mathcal{L}_1(y,z,u)+\mu(y) \mathcal{L}_1(x,z,u),\\
     \label{eq:rep 2} \mu(\mathcal L(x,y,z))(u)&=\mathcal{L}_2(y,z,\mu(x)(u))-\mu(x) \mathcal{L}_2(y,z,u)
   \end{eqnarray}
  for all $x,y,z\in A$ and $u\in V$.
  \end{defi}

\section{Ternary $F$-manifold algebras}\label{sec:ternary F-algebras and deformations}
In this section, we introduce the notion of ternary $F$-manifold algebras as a generalization of $F$-manifold algebras given in \cite{Dot} see also \cite{Liu&Sheng&Bai}. They are the generalization of $F$-manifold algebra  and  ternary Nambu-Poisson algebra.

\begin{defi} A  tuple $(A,\cdot,[\cdot,\cdot,\cdot])$ is called a ternary  $F$-manifold algebra if $(A,\cdot)$ is a commutative associative algebra and $(A,[\cdot,\cdot,\cdot])$ is a $3$-Lie algebra, such that for all $x_1,x_2,x_3,x_4,x_5\in A$,   holds
:
\begin{equation}\label{3-Leibnizator}
\mathcal L(x_1\cdot x_2,x_3, x_4,x_5)=x_1\cdot \mathcal L(x_2,x_3,x_4,x_5)+x_2\cdot \mathcal L(x_1,x_3,x_4,x_5),
\end{equation}
where $ \mathcal L(x_1,x_2,x_3,x_4)$ is the $3$-Leibnizator define by
\begin{equation}
 \mathcal L(x_1,x_2,x_3,x_4)=[x_{1},x_{2}, x_{3}\cdot x_{4}]- x_{3}\cdot[x_{1},x_{2},x_{4}]-[x_{1},x_{2},x_{3}]\cdot x_{4}.
\end{equation}

\end{defi}
\begin{ex}
Every ternary Nambu-Poisson algebras is a ternary $F$-manifold.
\end{ex}

\begin{defi}
Let $(A,\cdot_A,[\cdot,\cdot,\cdot]_A)$ and $(B,\cdot_B,[\cdot,\cdot,\cdot]_B)$ be two ternary $F$-manifold algebras. A  homomorphism between $A$ and $B$ is a linear map $f:A\rightarrow B$ such that
\begin{eqnarray}
 f(x_1\cdot_A x_2)&=&f(x_1)\cdot_B f(x_2),\\
  f[x_1,x_2,x_3]_A&=&[f(x_1),f(x_2),f(x_3)]_B,\quad \forall~x_1,x_2,x_3\in A.
\end{eqnarray}
\end{defi}

\begin{pro}\label{ex:direct sum of HMA}{
  Let $(A,\cdot_A,[\cdot,\cdot,\cdot]_A)$ and $(B,\cdot_B,[\cdot,\cdot,\cdot]_B)$ be two ternary  $F$-manifold algebras. Then $(A\oplus B,\cdot_{A\oplus B},[\cdot,\cdot,\cdot]_{A\oplus B})$ is an ternary $F$-manifold algebra, where the product $\cdot_{A\oplus B}$ and bracket $[\cdot,\cdot,\cdot]_{A\oplus B}$ are given by
  \begin{eqnarray*}
   ( x_1+ y_1) \cdot_{A\oplus B} (x_2+ y_2)&=& x_1\cdot_A x_2+y_1\cdot_B y_2,\\
   {[  ( x_1+ y_1), (x_2+ y_2), (x_3+ y_3)]_{A\oplus B}}&=&[x_1,x_2,x_3]_A+ [y_1, y_2,y_3]_B
  \end{eqnarray*}
  for all $x_i\in A,y_i\in B, \forall i=1,2,3.$}
\end{pro}
\begin{proof}
  It easy to check that $(A\oplus B,\cdot_{A\oplus B})$ is a commutative associative algebra and $(A\oplus B,[\cdot,\cdot,\cdot]_{A\oplus B})$ is a $3$-Lie algebra and  for all $x_i\in A,y_i\in B, ~~ i=1,\cdots,4$
  \begin{align*}
       \mathcal L_{A\oplus B}(x_1+y_1, x_2+y_2,x_3+y_3, x_4+y_4)
      =& \mathcal L_{A}(x_1, x_2,x_3, x_4)+\mathcal L_{ B}(y_1, y_2,y_3, x_4).
  \end{align*}
 Using the above identity, we have
  \begin{align*}
      &\mathcal L_{A\oplus B}((x_1+y_1)\cdot_{A\oplus B} (x_2+y_2),x_3+y_3, x_4+y_4,x_5+y_5)\\
      =&  \mathcal L_{A\oplus B}(x_1\cdot_A x_2+y_1\cdot_B y_2,x_3+y_3, x_4+y_4,x_5+y_5)\\
     =& \mathcal L_{A}(x_1\cdot_A x_2,x_3, x_4,x_5)+ \mathcal L_{B}(y_1\cdot_B y_2,y_3, y_4,y_5)\\
     =&x_1\cdot_A \mathcal L_A(x_2,x_3,x_4,x_5)+x_2\cdot_A \mathcal L_A(x_1,x_3,x_4,x_5)\\
     &+y_1\cdot_B \mathcal L_B(y_2,y_3,y_4,y_5)+y_2\cdot_B \mathcal L_B(y_1,y_3,y_4,y_5)\\
     =&(x_1+y_1)\cdot_{A\oplus B}\mathcal L_{A\oplus B}( x_2+y_2,x_3+y_3, x_4+y_4,x_5+y_5)\\
     &+(x_2+y_2)\cdot_{A\oplus B} \mathcal L_{A\oplus B}(x_1+y_1,x_3+y_3, x_4+y_4,x_5+y_5)
  \end{align*}

  Then, Eq. \eqref{3-Leibnizator} holds.
\end{proof}

\begin{pro}\label{ex:direct sum of HMA}{
  Let $(A,\cdot_A,[\cdot,\cdot,\cdot]_A)$ be a ternary  $F$-manifold algebras and $(B,\cdot_B)$ be commutative associative algebra. Then $(A\otimes B,\cdot_{A\otimes B},[\cdot,\cdot,\cdot]_{A\otimes B})$ is an ternary $F$-manifold algebra, where the product $\cdot_{A\otimes B}$ is defined by
  \begin{align*}
   &( x_1\otimes y_1) \cdot_{A\otimes B} (x_2\otimes y_2)= (x_1\cdot_A y_1)\otimes (x_2\cdot_B y_2),\\
   &[x_1\otimes y_1,x_2\otimes y_2,x_3\otimes y_3]_{A\otimes B}=[x_1,x_2,x_3]_A\otimes (y_1\cdot_B y_2 \cdot y_3 ),
  \end{align*}
  for all $x_i\in A,y_i\in B, \forall i=1,2,3.$}
\end{pro}
\begin{proof}
  It is obvious to check that $(A\otimes B,\cdot_{A\otimes B})$  is a commutative associative algebra and $(A\otimes B,[\cdot,\cdot,\cdot]_{A\otimes B})$ is a $3$-Lie algebra. For any $x_i,y_i\in A, i=1,\cdots,4$ we have
  \begin{align*}
     &\mathcal L_{A\otimes B}(x_1\otimes y_1,x_2\otimes y_2,x_3\otimes y_3,x_4\otimes y_4)\\
     =&[x_{1}\otimes y_1,x_{2}\otimes y_2, (x_{3}\otimes y_3)\cdot_{A\otimes B} (x_{4}\otimes y_4)]_{A\otimes B}- (x_{3}\otimes y_3)\cdot_{A\otimes B}[x_{1}\otimes y_1,x_{2}\otimes y_2,x_{4}\otimes y_4]\\
     &-[x_{1}\otimes y_1,x_{2}\otimes y_2,x_{3}\otimes y_3]\cdot_{A\otimes B} (x_{4}\otimes y_4) \\
    =& [x_{1}\otimes y_1,x_{2}\otimes y_2, (x_{3}\cdot_A x_4)\otimes (y_3\otimes_B y_4)]_{A\otimes B}- (x_{3}\otimes y_3)\cdot_{A\otimes B}([x_1,x_2,x_4]_A\otimes (y_1\cdot_B y_2 \cdot_B y_4 ))\\
     &-([x_1,x_2,x_3]_A\otimes (y_1\cdot_B y_2 \cdot_B y_3 ))\cdot_{A\otimes B} (x_{4}\otimes y_4) \\
     =& [x_{1},x_{2}, (x_{3}\cdot_A x_4)]\otimes(y_1\cdot_B y_2\cdot_B y_3 \cdot_B y_4))- (x_{3}\cdot_A[x_1,x_2,x_4]_A)\otimes( y_1\cdot_B y_2 \cdot_B y_3\cdot y_4 )\\
     &-([x_1,x_2,x_3]_A\cdot_A x_{4}\otimes (y_1\cdot_B y_2 \cdot_B y_3\cdot_B y_4 )\\
     =& \mathcal L_A(x_1,x_2,x_3,x_4)\otimes (y_1\cdot_B y_2 \cdot_B y_3\cdot_B y_4 ).
  \end{align*}
  Using the above equality, we obtain,
  \begin{align*}
      & \mathcal L_{A\otimes B}((x_1\otimes y_1)\cdot_{A\otimes B} (x_2\otimes y_2),(x_3\otimes y_3), (x_4\otimes y_4),(x_5\otimes y_5))\\
      =&\mathcal L_{A\otimes B}((x_1\cdot_A x_2)\otimes (y_1\cdot_B y_2),(x_3\otimes y_3), (x_4\otimes y_4),(x_5\otimes y_5))\\
      =&\mathcal L_A(x_1\cdot_A x_2,x_3,x_4,x_5)\otimes (y_1\cdot_B y_2\cdot_B y_3\cdot_B y_4\cdot_B y_5)\\
      =&(x_1\cdot_A \mathcal L_A(x_2,x_3,x_4,x_5)\otimes (y_1\cdot_B y_2\cdot_B y_3\cdot_B y_4\cdot_B y_5)\\
      &+x_2\cdot_A \mathcal L_A(x_1,x_3,x_4,x_5)\otimes (y_1\cdot_B y_2\cdot_B y_3\cdot_B y_4\cdot_B y_5)\\
     =& (x_1\otimes y_1)\cdot_{A\otimes B} \mathcal L_A(x_2,x_3,x_4,x_5)\otimes (y_2\cdot_B y_3\cdot_B y_4\cdot_B y_5)\\
      &+(x_2\cdot_B y_2)\cdot_{A\otimes B} \mathcal L_A(x_1,x_3,x_4,x_5)\otimes (y_1\cdot_B y_3\cdot_B y_4\cdot_B y_5)\\
      =&(x_1\otimes y_1)\cdot_{A\otimes B}\mathcal L_{A\otimes B}(x_2\otimes y_2,x_3\otimes y_2,x_4\otimes y_4,x_5\otimes y_5)\\
       &+(x_2\cdot_B y_2)\cdot_{A\otimes B} \mathcal L_{A\otimes B}(x_1\otimes y_1,x_3\otimes y_3,x_4\otimes y_4,x_5\otimes y_5).
  \end{align*}
    Then, Eq. \eqref{3-Leibnizator} holds and $(A\otimes B,\cdot_{A\otimes B},[\cdot,\cdot,\cdot]_{A\otimes B})$ is an ternary $F$-manifold algebra
\end{proof}
\begin{pro}\label{constructionmanifold}
Let $(A,\cdot,[\cdot,\cdot,\cdot])$ be a ternary $F$-manifold algebra. Fixed $a\in A$ and define a bracket $[\cdot,\cdot]_a:A\otimes A\to A$ by $[x,y]_a=[x,a,y]$ for all $x,y\in A$.
Then $(A,\cdot,[\cdot,\cdot]_a)$  is an $F$-manifold algebra.
\end{pro}
\begin{proof}
  Thanks to \cite{Pozhidaev}, we have $(A,[\cdot,\cdot]_a)$ is a Lie algebra.
  Now, we will show Eq.\eqref{eq:HM relation} for any $x,y,z,w\in A$,
observe that
  \begin{align*}
      \mathcal L_a(x,y,z)&=[x,y\cdot z]_a-[x,y]_a\cdot z-y\cdot [x,z]_a\\
      &=[x,a,y\cdot z]-[x,a,y]\cdot z-y\cdot [x,a,z]\\
      &=\mathcal L(x,a,y,z).
  \end{align*}
   Then, we have
   \begin{align*}
      & \mathcal L_a(x\cdot y,z,w)-x\cdot \mathcal L_a(y,z,w)+y\cdot \mathcal L_a(x,z,w)\\
       =& \mathcal L( x\cdot y,a,z,w)-x\cdot \mathcal L(y,a,z,w)+y\cdot \mathcal L(x,a,z,w)\\
       =&0.
   \end{align*}
Therefore,  $(A,\cdot,[\cdot,\cdot]_a)$  is an $F$-manifold algebra.
\end{proof}

\ 

Now we generalize the result given in \cite{Makhlouf&Amri}  to the notion of $F$-manifold. Let  $(A, [\cdot,\cdot])$ be a Lie algebra  and $\tau:A\rightarrow \mathbb K$ a linear map.
For any $x_1, x_2,x_3 \in A$, we define
the 3-ary bracket by
\begin{equation}\label{crochet_n}
[x_1,x_2,x_3]_\tau=\tau(x_1)\cdot[x_2, x_3] - \tau(x_2)\cdot[x_1, x_3]+\tau(x_3)\cdot[x_1, x_2].
\end{equation}
The map $\tau$ is called trace or $[\cdot,\cdot]$-trace if $\tau[\cdot,\cdot]:=0$.

\begin{thm}\cite{ams:ternary}
With the above notation and $\tau$ is a trace, $(A,[\cdot,\cdot,\cdot]_\tau)$ is $3$-Lie algebra.
\end{thm}
\begin{thm}\label{inducedthm}
 Let  $(A,\cdot, [\cdot,\cdot])$ be an $F$-manifold algebra and $\tau$ be  a trace. Then $(A,\cdot, [\cdot,\cdot,\cdot]_\tau)$ is a ternary $F$-manifold algebra if and only if
 \begin{equation}\label{condutioninduced}
     \tau(x_1\cdot x_2) \mathcal L(x_3,x_4,x_5)=\tau(x_2)x_1\cdot  \mathcal L(x_3,x_4,x_5)+\tau(x_1)x_2\cdot \mathcal L(x_3,x_4,x_5).
 \end{equation} 
\end{thm}
\begin{proof}
We have $(A,\cdot)$ commutative associative algebra and $(A, [\cdot,\cdot,\cdot]_\tau)$ is $3$-Lie algebra. For any $x_i, i=1,\cdots,5$ we have
\begin{align*}
    \mathcal L_\tau(x_1,x_2,x_3,x_4)=&[x_{1},x_{2}, x_{3}\cdot x_{4}]_\tau- x_{3}\cdot[x_{1},x_{2},x_{4}]_\tau-[x_{1},x_{2},x_{3}]_\tau\cdot x_{4}\\
    =&\tau(x_1)\cdot[x_2, x_3\cdot x_4] - \tau(x_2)\cdot[x_1, x_3\cdot x_4]+\tau(x_3\cdot x_4)\cdot[x_1, x_2]\\
    &- x_{3}\cdot (\tau(x_1)\cdot[x_2, x_4] - \tau(x_2)\cdot[x_1, x_4]+\tau(x_4)\cdot[x_1, x_2])\\
    &-(\tau(x_1)\cdot[x_2, x_3] - \tau(x_2)[x_1, x_3]+\tau(x_3)\cdot[x_1, x_2])\cdot x_4\\
    =&\tau(x_1)([x_2, x_3\cdot x_4]-x_{3}\cdot  [x_2, x_4]-[x_2, x_3]\cdot  x_4)\\
    &-\tau(x_2)([x_1, x_3\cdot x_4]-x_3\cdot [x_1, x_4]-[x_1, x_3]\cdot x_4)\\
    &+\tau(x_3\cdot x_4)[x_1, x_2]-\tau(x_4)x_{3}\cdot[x_1, x_2]-\tau(x_3)[x_1, x_2]\cdot x_4\\
    =&\tau(x_1) \mathcal L(x_2,x_3,x_4)-\tau(x_2)\mathcal L(x_1,x_3,x_4)\\
    &+\tau(x_3\cdot x_4)[x_1, x_2]-\tau(x_4)x_{3}\cdot[x_1, x_2]-\tau(x_3)[x_1, x_2]\cdot x_4. 
\end{align*}
 Using the above relation, we obtain
  \begin{align*}
     &\mathcal L_\tau(x_1\cdot x_2,x_3, x_4,x_5)- x_1\cdot \mathcal L_\tau(x_2,x_3,x_4,x_5)-x_2\cdot \mathcal L_\tau(x_1,x_3,x_4,x_5)\\
     =&     \tau(x_1\cdot x_2) \mathcal L(x_3,x_4,x_5)-\tau(x_3)\mathcal L(x_1\cdot x_2,x_4,x_5)-\tau(x_2)x_1\cdot  \mathcal L(x_3,x_4,x_5)\\
      &-\tau(x_3)x_1\cdot\mathcal L(x_2,x_4,x_5))-\tau(x_1) x_2\cdot\mathcal L(x_3,x_4,x_5)-\tau(x_3)x_2\cdot\mathcal L(x_1,x_4,x_5)\\
      &+\tau(x_4\cdot x_5) [x_1\cdot x_2,x_3]-\tau(x_5)x_{4}\cdot[x_1\cdot x_2, x_3]-\tau(x_4)[x_1\cdot x_2,x_3]\cdot x_5\\ 
      &-\tau(x_4\cdot x_5)x_1\cdot[x_2, x_3]+\tau(x_5)x_1\cdot x_{4}\cdot[x_2, x_3]+\tau(x_4)x_1\cdot [x_2, x_3]\cdot x_5\\ 
      &-\tau(x_4\cdot x_5)x_2\cdot[x_1, x_3]+\tau(x_5)x_2\cdot x_{4}\cdot[x_1, x_3]+\tau(x_4)x_2\cdot [x_1, x_3]\cdot x_5\\ 
      =& \tau(x_1\cdot x_2) \mathcal L(x_3,x_4,x_5)-\tau(x_3)\mathcal L(x_1\cdot x_2,x_4,x_5)-\tau(x_2)x_1\cdot  \mathcal L(x_3,x_4,x_5)\\
      &-\tau(x_3)x_1\cdot\mathcal L(x_2,x_4,x_5))-\tau(x_1) x_2\cdot\mathcal L(x_3,x_4,x_5)-\tau(x_3)x_2\cdot\mathcal L(x_1,x_4,x_5))\\
      &+\tau(x_4\cdot x_5) \mathcal L(x_1,x_2,x_3)-\tau(x_5)x_{4}\cdot \mathcal L(x_1,x_2,x_3)-\tau(x_4) \mathcal L(x_1,x_2,x_3)\cdot x_{5}\\
      =&\tau(x_1\cdot x_2) \mathcal L(x_3,x_4,x_5)-\tau(x_1) x_2\cdot\mathcal L(x_3,x_4,x_5)-\tau(x_2)x_1\cdot  \mathcal L(x_3,x_4,x_5) \\
      &+\tau(x_4\cdot x_5) \mathcal L(x_1,x_2,x_3)-\tau(x_4) \mathcal L(x_1,x_2,x_3)\cdot x_{5}-\tau(x_5)x_{4}\cdot \mathcal L(x_1,x_2,x_3).
        \end{align*}
Then $(A,\cdot, [\cdot,\cdot,\cdot]_\tau)$ is a ternary $F$-manifold algebra if and only if  Eq. \eqref{condutioninduced} holds.
\end{proof}

\begin{cor}    
Let $(A,\cdot,[\cdot,\cdot])$ be a $F$-manifold algebra and $\tau$ a trace. 
If $(A,\cdot)$  is unital with unit $1_A$ and  
 $$\tau(x\cdot y)1_A=\tau(x) y+\tau( y)x.$$
Then $(A,\cdot,[\cdot,\cdot,\cdot]_\tau)$ is ternary $F$-manifold algebra.  
\end{cor}

\section{Representations of ternary $F$-manifold algebras}\label{representationternarymanifold}

Let $(A,\cdot ,[\cdot,\cdot,\cdot])$ be a  ternary $F$-manifold algebra, $(V;\rho)$ be a
  representation of the $3$-Lie algebra $(A,[\cdot,\cdot,\cdot])$ and $(V;\mu)$ be a representation of the commutative associative algebra $(A,\cdot)$. Define the three linear maps $\mathcal{L}_1,\mathcal{L}_2, \mathcal{L}_3:A\otimes A\otimes A\otimes V\to V$
  given by
  \begin{eqnarray}
  \label{eq:repH 1}\mathcal{L}_1(x,y,z,u)&=&\rho(x,y)\mu(z)(u)-\mu(z) \rho(x,y)(u)-\mu([x,y,z])u,\\
  \label{eq:repH 2}\mathcal{L}_2(x,y,z,u)&=&\mu(z)\rho(x,y)(u)+\mu(y) \rho(x,z)(u)-\rho(x,y\cdot z)u,\\
\label{eq:repH 3} \mathcal{L}_3(x,y,z,u)&=&\rho(x,y)  \mu(z)(u)+\rho(x,z)  \mu(y)(u)-\rho(x,y\cdot  z)u
   \end{eqnarray}
  for all $x,y\in A$ and $u\in V$. Note that, if we consider $V=A$, then $\mathcal{L}_1=\mathcal{L}$ and $\mathcal{L}_2=\mathcal{L}\sigma$, where $\sigma(x,y,z)=(z,x,y)$.
\begin{defi}\label{def-rep-manifo}With the above notations, the tuple $(V;\rho,\mu)$ is a representation of $A$  if the following conditions hold:
   \begin{eqnarray}
     \label{eq:rep 1}\mathcal{L}_1(x_1\cdot  x_2,x_3,x_4,u)=\mu(x_1) \mathcal{L}_1(x_2,x_3,x_4,u)+\mu(x_2) \mathcal{L}_1(x_1,x_3,x_4,u),\\
       \label{eq:rep 3}\mathcal{L}_2(x_1\cdot  x_2,x_3,x_4,u)=\mu(x_1) \mathcal{L}_2(x_2,x_3,x_4,u)+\mu(x_2) \mathcal{L}_2(x_1,x_3,x_4,u),\\
     \label{eq:rep 2} \mu(\mathcal L(x_1,x_2,x_3,x_4))u=\mathcal{L}_2(x_2,x_3,x_4,\mu(x_1)(u))-\mu(x_1) \mathcal{L}_2(x_2,x_3,x_4,u),
   \end{eqnarray}
  for all $x_1,x_2,x_3,x_4\in A$ and $u\in V$.
  \end{defi}
  \begin{ex}
  Let $(A,\cdot ,[\cdot,\cdot,\cdot])$ be a  ternary $F$-manifold algebra. Then $(A,\ad, L)$ is a representation of $A$ called the adjoint representation.
  \end{ex}
\begin{ex}
 Let $(V;\rho,\mu)$ be a representation of a ternary Nambu-Poisson algebra $(P,\cdot,[\cdot,\cdot,\cdot])$, i.e. $(V;\rho,\phi)$ is a  representation of the $3$-Lie
 algebra $(P,[\cdot,\cdot,\cdot])$ and $(V;\mu)$ is a representation of the commutative  associative algebra $(P,\cdot)$ satisfying
   \begin{eqnarray*}
          &&\mathcal{L}_1(x_1,x_2,x_3,u)=\mathcal{L}_2(x_1,x_2,x_3,u)  =0,\quad \forall~x_1,x_2,x_3\in P, u\in V.
   \end{eqnarray*}
  Then $(V;\rho,\mu)$ is also a representation of the  ternary $F$-manifold algebra given by this  ternary Nambu-Poisson algebra $P$.
\end{ex}
\begin{ex}
 Let $(A,\cdot ,[\cdot,\cdot,\cdot])$ be a  ternary $F$-manifold algebra and $(V;\rho,\mu)$ be a representation of $A$. Fixed $a\in A$, we define a linear map $\rho_a:A\to End(V)$ defined by $\rho_a(x)v=\rho(x,a)v$. Then, $(V;\rho_a,\mu)$ be a representation of the $F$-manifold algebra defined in Proposition \ref{constructionmanifold}.
\end{ex}

\begin{pro}\label{pro:semi-direct}
 Let $(A,\cdot ,[\cdot,\cdot,\cdot])$ be a  ternary $F$-manifold algebra. Then $(V;\rho,\mu)$ is a representation  of $A$ if and only if $(A\oplus V,\cdot_{\mu},[\cdot,\cdot,\cdot]_\rho)$ is
 a   ternary $F$-manifold algebra, where $(A\oplus V,\cdot_{\mu})$ is the semi-direct product commutative associative algebra $A\ltimes_{\mu} V$, i.e.
 $$
  (x_1+v_1)\cdot_{\mu}(x_2+v_2)=x_1\cdot  x_2+\mu(x_1)v_2+\mu(x_2)v_1,\quad
$$

  and $(A\oplus V,[\cdot,\cdot,\cdot]_\rho)$ is the semi-direct product $3$-Lie algebra $A\ltimes_{\rho} V$, i.e. $$
  [x_1+v_1,x_2+v_2,x_3+v_3]_\rho=[x_1,x_2,x_3]+\rho(x_1,x_2)v_3-\rho(x_1,x_3)v_2+\rho(x_2,x_3)v_1,$$
  $for all~x_1,x_2,x_3\in A,~v_1,v_2,v_3\in V.$

\end{pro}
\begin{proof}
  It is obvious to check that $(A\oplus V,\cdot_{\mu})$ is commutative associative algebra and $(A\oplus V,[\cdot,\cdot,\cdot]_\rho)$ is a $3$-Lie algebra. Observe that for any $x_i\in A$ and $u_i\in V$, $i=1,\cdots,4$
  \begin{align*}
    \mathcal{L}_{A\oplus V}(x_1+u_1,x_2+u_2,x_3+u_3,x_4+u_4)=& \mathcal{L}(x_1,x_2,x_3,x_4)+\mathcal{L}_{1}(x_1,x_2,x_3,u_4)+\mathcal{L}_{1}(x_1,x_2,x_4,u_3)\\ &+\mathcal{L}_{2 }(x_1,x_3,x_4,u_2)-\mathcal{L}_{2}(x_2,x_3,x_4,u_1).
  \end{align*}
  So, we have
  \begin{align*}
       &\mathcal L_{A\oplus V}((x_1+u_1)\cdot_{A\oplus V} (x_2+u_2),x_3+u_3, x_4+u_4,x_5+u_5)\\
      &-(x_1+u_1)\cdot_{A\oplus V} \mathcal L_{A\oplus V}(x_2+u_2,x_3+u_3,x_4+u_4,x_5+u_5)\\
       &-(x_2+u_2)\cdot_{A\oplus V} \mathcal L_{A\oplus V}(x_1+u_1,x_3+u_3,x_4+u_4,x_5+u_5)\\
       =&\mathcal L_{A\oplus V}(x_1\cdot x_2+\mu(x_1)u_2+\mu(x_2)u_1,x_3+u_3, x_4+u_4,x_5+u_5)\\
       &-(x_1+u_1)\cdot_{A\oplus V} \big(\mathcal{L}(x_2,x_3,x_4,x_5)+\mathcal{L}_{1}(x_2,x_3,x_4,u_5)\\
       &+\mathcal{L}_{1}(x_2,x_3,x_5,u_4)+\mathcal{L}_{2 }(x_2,x_4,x_5,u_3)-\mathcal{L}_{2}(x_3,x_4,x_5,u_2)\big)\\
    &-(x_2+u_2)\cdot_{A\oplus V} \big(\mathcal{L}(x_1,x_3,x_4,x_5)+\mathcal{L}_{1}(x_1,x_3,x_4,u_5)\\
       &+\mathcal{L}_{1}(x_1,x_3,x_5,u_4)+\mathcal{L}_{2 }(x_1,x_4,x_5,u_3)-\mathcal{L}_{2}(x_3,x_4,x_5,u_1)\big)\\
       =&\mathcal{L}(x_1\cdot x_2,x_3,x_4,x_5)+\mathcal{L}_{1}(x_1\cdot x_2,x_3,x_4,u_5)+\mathcal{L}_{1}(x_1\cdot x_2,x_3,x_5,u_4)\\
       &+\mathcal{L}_{2 }(x_1\cdot x_2,x_4,x_5,u_3)-\mathcal{L}_{2}(x_3,x_4,x_5,\mu(x_1)u_2)-\mathcal{L}_{2}(x_3,x_4,x_5,\mu(x_2)u_1)\\
       &-x_1\cdot \mathcal{L}(x_2,x_3,x_4,x_5)-\mu(\mathcal{L}(x_2,x_3,x_4,x_5))u_1-\mu(x_1)\mathcal{L}_{1}(x_2,x_3,x_4,u_5)\\
      & - \mu(x_1)\mathcal{L}_{1}(x_2,x_3,x_5,u_4)-\mu(x_1)\mathcal{L}_{2 }(x_2,x_4,x_5,u_3)+\mu(x_1)\mathcal{L}_{2}(x_3,x_4,x_5,u_2)\\
       &-x_2\cdot \mathcal{L}(x_1,x_3,x_4,x_5)-\mu(\mathcal{L}(x_1,x_3,x_4,x_5))u_2-\mu(x_2)\mathcal{L}_{1}(x_1,x_3,x_4,u_5)\\
       &-\mu(x_2)\mathcal{L}_{1}(x_1,x_3,x_5,u_4)-\mu(x_2)\mathcal{L}_{2 }(x_1,x_4,x_5,u_3)+\mu(x_2)\mathcal{L}_{2}(x_3,x_4,x_5,u_1).
  \end{align*}
   Then, $(V;\rho,\mu)$ is a representation  of $A$ if and only if $(A\oplus V,\cdot_{\mu},[\cdot,\cdot,\cdot]_\rho)$ is
 a   ternary $F$-manifold algebra.
\end{proof}

Let $(V;\rho,\mu)$ be a representation of a ternary Nambu-Poisson algebra $(P,\cdot_P,\{\cdot,\cdot,\cdot\}_P)$. Then,  $(V^*;\rho^\ast,\mu^\ast)$ is also a representation of $P$. But  ternary $F$-manifold algebras do not have this property. In fact, we have
\begin{lem}
  Let $(A,\cdot ,[\cdot,\cdot,\cdot])$ be a  ternary $F$-manifold algebra and     $(V;\rho,\mu)$ be a representation of $A$.  Define the two linear maps $\mathcal{L}_1^\ast,\mathcal{L}_2^\ast:A\otimes A\otimes A\otimes V^*\to V^*$
  given by
  \begin{eqnarray*}
  \label{eq:repH 1}\mathcal{L}_1^\ast(x,y,z,\xi)&=&-\rho^\ast(x,y)\mu^\ast(z)(\xi)+\mu^\ast(z) \rho^\ast(x,y)(\xi)+\mu^\ast([x,y,z])(\xi),\\
  \label{eq:repH 2}\mathcal{L}_2^\ast(x,y,z,\xi)&=&-\mu^\ast(x)\rho^\ast(y,z)(\xi)-\mu^\ast(y) \rho^\ast(x,z)(\xi)+\rho^\ast(x,y\cdot z)(\xi)
   \end{eqnarray*}
  for all $x,y,z\in A$ and $\xi\in V^*$. Then we have
  $$\langle\mathcal{L}_1^\ast(x,y,z,\xi),u\rangle=\langle\xi,\mathcal{L}_1(x,y,z,u)\rangle$$
  and
  $$\langle\mathcal{L}_2^\ast(x,y,z,\xi),u\rangle=-\langle\xi,\mathcal{L}_3(x,y,z,u)\rangle.$$
  for all $x,y,z\in A$ and $\xi\in V^\ast$.
\end{lem}
    \begin{proof}
Let $x,y,z\in A$ and $\xi\in V^\ast$, we have 
\begin{align*}
    \langle\mathcal{L}_1^\ast(x,y,z,\xi),u\rangle&= \langle -\rho^\ast(x,y)\mu^\ast(z)(\xi)+\mu^\ast(z) \rho^\ast(x,y)(\xi)+\mu^\ast([x,y,z])(\xi),u\rangle\\
    &=-\langle \rho^\ast(x,y)\mu^\ast(z)(\xi),u\rangle+\langle\mu^\ast(z) \rho^\ast(x,y)(\xi),u\rangle+\langle \mu^\ast([x,y,z])(\xi),u\rangle\\
    &=-\langle \xi,\mu(z)\rho(x,y)u\rangle+\langle\xi,\rho(x,y)\mu(z) (u)\rangle-\langle \xi,\mu([x,y,z])(u)\rangle\\
    &=\langle\xi,\mathcal{L}_1(x,y,z,u)\rangle.
\end{align*}
Similarly, we obtain  $$\langle\mathcal{L}_2^\ast(x,y,z,\xi),u\rangle=-\langle\xi,\mathcal{L}_3(x,y,z,u)\rangle.$$
    \end{proof}
    
\begin{pro}
  Let $(A,\cdot ,[\cdot,\cdot,\cdot])$ be a ternary-$F$-manifold algebra. If    the tuple $(V;\rho,\mu)$ is representation of $A$ satisfying the following axioms
  \begin{eqnarray}
     \label{eq:corep 1}&&\mathcal{L}_1(x\cdot  y,z,t,u)= \mathcal{L}_1(y,z,t,  \mu(x)u)+ \mathcal{L}_1(x,z,t,  \mu(y)u),\\
   \label{eq:corep 2}&&\mathcal{L}_3(x\cdot  y,z,t,u)=- \mathcal{L}_3(y,z,t,  \mu(x)u)- \mathcal{L}_3(x,z,t,  \mu(y)u),\\ 
     \label{eq:corep 3} &&\mu(\mathcal L(x,y,z,t))u  =\mathcal{L}_3(y,z,t,  \mu(x)u)-\mu(x)  \mathcal{L}_3(y,z,t,u)
   \end{eqnarray}
   for any $x,y,z,t \in A$ and $u \in V$. 
   Then $(V^*;\rho^\ast,-\mu^\ast)$ is  a representation of $A$ called the dual representation. 
\end{pro}
\begin{proof}
For any $x,y,z,t\in A,v\in u,\xi\in V^*$, we have
\begin{align*}
 &\langle\mathcal{L}_1^\ast(x\cdot  y,z,t,\xi)+\mu^\ast(x) \mathcal{L}_1^\ast(y,z,t,\xi)+\mu^\ast(y) \mathcal{L}_1^\ast(x,z,t,\xi),u\rangle\\
 =&\langle \xi,\mathcal{L}_1(x\cdot y,z,t,u)- \mathcal{L}_1(y,z,t,\mu(x)u)- \mathcal{L}_1(x,z,t,\mu(y)u)\rangle,
\end{align*}
\begin{align*}
 &\langle\mathcal{L}_2^\ast(x\cdot  y,z,t,\xi)+\mu^\ast(x) \mathcal{L}_2^\ast(y,z,t,\xi)+\mu^\ast(y) \mathcal{L}_2^\ast(x,z,t,\xi),u\rangle\\
 =&\langle \xi,-\mathcal{L}_3(x\cdot y,z,t,u)+ \mathcal{L}_3(y,z,t,\mu(x)u)+ \mathcal{L}_1(x,z,t,\mu(y)u)\rangle
\end{align*}
and
\begin{align*}
    & \langle -\mu^\ast(\mathcal{L}(x,y,z,t)\xi+\mathcal{L}_2^\ast(y,z,t,\mu^\ast(x)\xi-\mu^\ast(x)\mathcal{L}_2^\ast(y,z,t,\xi)
    ,u\rangle\\
 =& \langle \xi, \mu(\mathcal{L}(x,y,z,t))u +\mu(x)\mathcal{L}_3(y,z,t,u)-\mathcal{L}_3(y,z,t,\mu(x)u) \rangle .
\end{align*}

According to  Eqs. \eqref{eq:corep 1}-- \eqref{eq:corep 3} and  Definition \ref{def-rep-manifo}, the conclusion follows immediately.
\end{proof}

\begin{cor}\label{ex:dual representation}
  Let $(A,\cdot ,[\cdot,\cdot,\cdot])$ be a ternary-$F$-manifold algebra such that  the following relations hold:
  \begin{align}  
  &\label{eq:coh1}
   \mathcal{L} (x\cdot  y,z,t,u)=\mathcal{L} (y,z,t,x\cdot  u)+\mathcal{L} (x,z,t,y\cdot  u),\\
 & \mathcal{K}(x\cdot y,z,t\cdot u)=-
 \mathcal{K}(x,z,t,y\cdot u)-\mathcal{K}(y,z,t,x\cdot u), \label{eq:coh2}\\
   &  \mathcal{L} (x, y,z,t)\cdot u=\mathcal K(y,z,t,x\cdot  u)-x \cdot  \mathcal K(y,z,t,u)  \label{eq:coh3}
  \end{align}
     for all $x,y,z,t,u\in A$, 
 where $\mathcal K:\otimes^4 A\rightarrow
A$ is defined by
  $$\mathcal{K}(x,y,z,u)=\circlearrowleft_{y,z,u}[x,y,z\cdot u]. $$

Then $(A^*;\ad^\ast,-L^\ast)$ is a representation of $A$ called the coadjoint representation.
\end{cor}

\begin{defi}
  A  coherence ternary-$F$-manifold algebra is a ternary-$F$-manifold algebra such that Eqs. \eqref{eq:coh1}--\eqref{eq:coh3} hold.
\end{defi}

\section{Ternary pre-$F$-manifold algebras}\label{preF}

In the following,  we give a notion of ternary pre-$F$-manifold algebras which is a generalization of  pre-$F$-manifold algebra (see \cite{Liu&Sheng&Bai}) and give some results.

\subsection{Definition and elementary results}
\begin{defi}\label{ternary-preP}
  A  ternary pre-$F$-manifold algebra is a triple $(A,\diamond,\{\cdot,\cdot,\cdot\})$, where $(A,\diamond)$ is a Zinbiel algebra and $(A,\{\cdot,\cdot,\cdot\})$ is a $3$-pre-Lie algebra, such that:
  \begin{align}
    \label{eq:pre-M 1}&F_1(x_1\cdot x_2,x_3,x_4,x_5)=x_1\diamond F_1(x_2,x_3,x_4,x_5)+x_2\diamond F_1(x_1,x_3,x_4,x_5),\\
     \label{eq:pre-M 11}&F_2(x_1\cdot x_2,x_3,x_4,x_5)=x_1\diamond F_1(x_2,x_3,x_4,x_5)+x_2\diamond F_1(x_1,x_3,x_4,x_5),\\
    \label{eq:pre-M 2}&\mathcal L(x_1,x_2,x_3,x_4)\diamond  x_5=F_2(x_2,x_3,x_4,x_1\diamond x_5)-x_1\diamond F_2(x_2,x_3,x_4,x_5)
  \end{align}
 where $F_1,F_2,\mathcal L:\otimes^4 A\longrightarrow A$ are defined by
  \begin{eqnarray}\label{F1}
    F_1(x_1,x_2,x_3,x_4)&=&\{x_1,x_2,x_3\diamond x_4\}-x_3\diamond\{x_1,x_2,x_4\} -[x_1,x_2,x_3]\diamond x_4,\\\label{F2}
    F_2(x_1,x_2,x_3,x_4)&=&x_3\diamond\{x_1,x_2,x_4\}+x_2\diamond\{x_1,x_3,x_4\}-\{x_1,x_2\cdot x_3,x_4\},\\\label{Leib}
     \mathcal L(x_1,x_2,x_3,x_4)&=&[x_{1},x_{2}, x_{3}\cdot x_{4}]- x_{3}\cdot[x_{1},x_{2},x_{4}]-[x_{1},x_{2},x_{3}]\cdot x_{4}
  \end{eqnarray}
and the operation $\cdot$ and the bracket $[\cdot,\cdot,\cdot]$ are defined by
\begin{equation}\label{eq:pHM-operations}
  x\cdot y=x\diamond y+y\diamond x,\quad [x,y,z]=\circlearrowleft_{x,y,z}\{x,y,z\},
\end{equation}
for all $x,y,z,\in A$.
\end{defi}
If $F_1=F_2=0$ in the above definition, then we obtain a  particular case of ternary pre-$F$-manifold algebras. It will be called ternary pre-Nambu-Poisson algebra since it can be obtained by splitting of a ternary  Nambu-Poisson algebra  introduced in \cite{Makhlouf&Amri}. 
\begin{defi}
 A  ternary pre-Nambu-Poisson algebra is a tuple $(A,\diamond,\{\cdot,\cdot,\cdot\})$, where $(A,\diamond)$ is a Zinbiel algebra and $(A,\{\cdot,\cdot,\cdot\})$ is a $3$-pre-Lie algebra, such that:
  \begin{eqnarray*}
  \{x_1,x_2,x_3\diamond x_4\}&=&x_3\diamond\{x_1,x_2,x_4\} +[x_1,x_2,x_3]\diamond x_4,\\
    \{x_1,x_2\cdot x_3,x_4\}&=&x_3\diamond\{x_1,x_2,x_4\}+x_2\diamond\{x_1,x_3,x_4\}
  \end{eqnarray*}
for all $x_1,x_2,x_3,x_4\in A.$

\end{defi}

\begin{thm}
Let   $(A,\diamond,\{\cdot,\cdot,\cdot\})$  be a ternary pre-$F$-manifold algebra. Then
   $(A,\cdot,[\cdot,\cdot,\cdot])$ is a  ternary $F$-manifold algebra, where the operation $\cdot$ and bracket $[\cdot,\cdot,\cdot]$ are given by Eq. \eqref{eq:pHM-operations}, which is called the   sub-adjacent
ternary $F$-manifold algebra  of $(A,\diamond,\{\cdot,\cdot,\cdot\})$  and denoted by $A^c$.

\end{thm}
\begin{proof}
 Thanks to  Lemma  \ref{lem:den-ass} and Proposition \ref{3preLieTo3Lie},  we have $(A,\cdot)$ is a commutative associative algebra and $(A,[\cdot,\cdot,\cdot])$ is a $3$-Lie algebra.  
Using Eqs. \eqref{F1}-\eqref{Leib}, we have
  \begin{eqnarray*}
      \mathcal L(x_1,x_2,x_3,x_4)&=&[x_{1},x_{2}, x_{3}\cdot x_{4}]- x_{3}\cdot[x_{1},x_{2},x_{4}]-[x_{1},x_{2},x_{3}]\cdot x_{4}\\
      &=&\{x_{1},x_{2}, x_{3}\cdot x_{4}\}+\{x_{2},x_{3}\cdot x_{4},x_{1}\}
      +\{x_{3}\cdot x_{4},x_{1},x_{2},\}\\
      &&- x_{3}\diamond [x_{1},x_{2},x_{4}]- [x_{1},x_{2},x_{4}]\diamond x_{3}-[x_{1},x_{2},x_{3}]\diamond x_{4}- x_{4}\diamond [x_{1},x_{2},x_{3}]\\
      &&- x_{3}\diamond [x_{1},x_{2},x_{4}]- [x_{1},x_{2},x_{4}]\diamond x_{3}-[x_{1},x_{2},x_{3}]\diamond x_{4}- x_{4}\diamond [x_{1},x_{2},x_{3}]\\
       &=&\{x_{1},x_{2}, x_{3}\diamond x_{4}\}+\{x_{1},x_{2}, x_{4}\diamond x_{3}\}+\{x_{2},x_{3}\cdot x_{4},x_{1}\}
      -\{x_{1},x_{3}\cdot x_{4},x_{2}\}\\
      &&- x_{3}\diamond \{x_{1},x_{2},x_{4}\}- x_{3}\diamond \{x_{2},x_{4},x_{1}\}- x_{3}\diamond \{x_{4},x_{1},x_{2}\}\\
     && -x_{4}\diamond \{x_{1},x_{2},x_{3}\}-x_{4}\diamond \{x_{2},x_{3},x_{1}\}-x_{4}\diamond \{x_{3},x_{1},x_{2}\}\\
     &&- [x_{1},x_{2},x_{4}]\diamond x_{3}-[x_{1},x_{2},x_{3}]\diamond x_{4} \\
     &=&F_1(x_{1},x_{2},x_{3},x_{4})+F_1(x_{1},x_{2},x_{4},x_{3})+F_2(x_{1},x_{3},x_{4},x_{2})-F_2(x_{2},x_{3},x_{4},x_{1}).
  \end{eqnarray*}
    According to Eqs.\eqref{eq:pre-M 1}, \eqref{eq:pre-M 11} and  \eqref{eq:pre-M 2}, we get
\begin{eqnarray*}
 && \mathcal L(x_1\cdot x_2,x_3, x_4,x_5)-x_1\cdot \mathcal L(x_2,x_3,x_4,x_5)-x_2\cdot \mathcal L(x_1,x_3,x_4,x_5)\\
  &=&F_1(x_1\cdot x_2,x_3, x_4,x_5)+F_1(x_1\cdot x_2,x_3, x_5,x_4)+F_2(x_1\cdot x_2,x_4, x_5,x_3)-F_2(x_3, x_4,x_5,x_1\cdot x_2)\\
 && -x_1\diamond \mathcal L(x_2,x_3,x_4,x_5)
  -\mathcal L(x_2,x_3,x_4,x_5)\diamond x_1 -x_2\diamond \mathcal L(x_1,x_3,x_4,x_5)
  -\mathcal L(x_1,x_3,x_4,x_5)\diamond x_2\\
  &=&F_1(x_1\cdot x_2,x_3, x_4,x_5)+F_1(x_1\cdot x_2,x_3, x_5,x_4)+F_2(x_1\cdot x_2,x_4, x_5,x_3)\\
  &&-F_2(x_3, x_4,x_5,x_1\diamond x_2)  -F_2(x_3, x_4,x_5,x_2\diamond x_1)  -x_1\diamond F_1(x_2,x_3,x_4,x_5)-x_1\diamond F_1(x_2,x_3,x_5,x_4)\\
  &&-x_1\diamond F_2(x_2,x_4,x_5,x_3) +x_1\diamond F_2(x_3,x_4,x_5,x_2) -\mathcal L(x_2,x_3,x_4,x_5)\diamond x_1 -x_2\diamond F_1(x_1,x_3,x_4,x_5)\\
  &&-x_2\diamond F_1(x_1,x_3,x_5,x_4)-x_2\diamond F_2(x_1,x_4,x_5,x_3)+x_2\diamond F_2(x_3,x_4,x_5,x_1)
  -\mathcal L(x_1,x_3,x_4,x_5)\diamond x_2\\
  &=&0.
\end{eqnarray*}
Therefore,  $(A,\cdot,[\cdot,\cdot,\cdot])$ is a  ternary $F$-manifold algebra.
\end{proof}

The following observation is a direct consequence of the above Theorem.
\begin{cor}
Let   $(A,\diamond,\{\cdot,\cdot,\cdot\})$  be a ternary pre-Nambu-Poisson algebra. Then
   $(A,\cdot,[\cdot,\cdot,\cdot])$ is a  ternary Nambu-Poisson algebra, where the operation $\cdot$ and bracket $[\cdot,\cdot,\cdot]$ are given by Eq. \eqref{eq:pHM-operations}.
\end{cor}
\begin{pro} The tuple 
$(A;\mathbb{L},L_\diamond)$ is a representation of the sub-adjacent ternary-$F$-manifold algebras $A^c$, where $\mathbb{L}$ and $L_\diamond$ are given by Eqs. \eqref{ed 3 pre lie}  and \eqref{eq:dendriform-rep}, respectively.
\end{pro}
\begin{proof}
Thanks to  Lemma  \ref{lem:den-ass} and Lemma \ref{adjoint 3prelie},  $(A;L_\diamond)$ is a representation of the commutative associative algebra $(A,\cdot)$ as well as $(A;\mathbb{L})$ is a representation of the sub-adjacent $3$-Lie algebra $A^c$. 

According to  Eqs. \eqref{eq:pre-M 1}-\eqref{eq:pre-M 2} we can easly check  Eqs. \eqref{eq:rep 1}-\eqref{eq:rep 2}.  Therefore,  $(A;\mathbb{L},L_{\diamond})$ is a representation of the sub-adjacent ternary-$F$-manifold algebra $A^c$.
\end{proof}

\subsection{Relative Rota-Baxter operators  on ternary-$F$-manifold algebras}
The notion of relative Rota-Baxter operators (called also $\mathcal{O}$-operators) was first given for Lie algebras by Kupershmidt in \cite{K} as a natural generalization of the classical Yang-Baxter equation and then defined by analogy in other various (associative, alternative, Jordan, ...). In \cite{Bai&Guo&Sheng}, authors extended this notion to $3$-Lie algebras.  

A linear map $T:V\longrightarrow A$ is called a relative Rota-Baxter operator on a commutative associative algebra $(A,\cdot)$ with respect to a representation  $(V;\mu)$ if $T$ satisfies
  \begin{align}
    Tu\cdot  Tv&=T(\mu(Tu)v+\mu(Tv)u),\quad\forall~u,v\in V.
  \end{align}
In particular, a Relative Rota-Baxter operator on a commutative associative
algebra $(A,\cdot)$ with respect to the adjoint representation
is called a Rota-Baxter operator of weight zero or briefly a  Rota-Baxter operator on $A$.

 It is obvious to obtain the following result.
\begin{lem}\label{Zinbiel}
  Let $(A,\cdot)$ be a commutative associative algebra and $(V;\mu)$ a representation. Let $T:V\rightarrow A$ be a relative Rota-Baxter operator on  $(A,\cdot)$ with respect to $(V;\mu)$. Then there exists a Zinbiel algebra structure on $V$ given by
$$
    u\diamond v=\mu(Tu)v, \quad\forall~u,v\in V.
$$
\end{lem}

A linear map $T:V\longrightarrow   A$ is called a relative Rota-Baxter  operator on a $3$-Lie algebra $(A,[\cdot,\cdot,\cdot])$ with respect to a representation $(V;\rho)$ if $T$ satisfies
\begin{align}
  [Tu, Tv,Tw] &=T\Big(\rho(Tu,Tv)w+\rho(Tv,Tw)u+\rho(Tw,Tu)v\Big),\quad \forall~u,v,w\in V.\label{O-operator2}
\end{align}
In particular, a relative Rota-Baxter operator on a $3$-Lie algebra
$(  A,[\cdot,\cdot,\cdot] )$ with respect to the adjoint representation $ad$ is called a 
Rota-Baxter operator of weight zero or briefly a  Rota-Baxter operator on $  A$.

\begin{lem}{\rm(\cite{Bai&Guo&Sheng})}\label{3-Lie}
Let $T:V\to   A$ be a relative Rota-Baxter operator  on a $3$-Lie algebra $(  A,[\cdot,\cdot,\cdot] )$ with respect to a representation $(V;\rho)$. Define the bracket  $\{\cdot,\cdot,\cdot\}$ on $V$ by
\begin{equation}
  \{u, v,w \}=\rho(Tu,Tv)w,\quad \forall~u,v,w\in V.
\end{equation}
Then $(V,\{\cdot,\cdot,\cdot\})$ is a $3$-pre-Lie algebra.
 \end{lem}

Let $(V;\rho,\mu)$ be a representation of a ternary-$F$-manifold algebra $(A,\cdot ,[\cdot,\cdot,\cdot] )$.
\begin{defi}
 A linear operator $T:V\longrightarrow A$ is called  a  relative Rota-Baxter operator on $(A,\cdot ,[\cdot,\cdot,\cdot] )$ if $T$ is both a relative Rota-Baxter operator on the commutative associative algebra $(A,\cdot)$ and a relative Rota-Baxter operator on the $3$-Lie algebra $(A,[\cdot,\cdot,\cdot] )$.
 \end{defi}
In particular,   a linear operator $\mathcal R:A\longrightarrow A$ is
called a 
Rota-Baxter operator of weight zero or briefly a  Rota-Baxter
operator on $A$, if $\mathcal R$ is both a Rota-Baxter operator on
the commutative associative algebra $(A,\cdot )$ and a
Rota-Baxter operator on the $3$-Lie algebra $(A,[\cdot,\cdot,\cdot])$.

In the following result, we give the construction of ternary pre-$F$-manifold algebra using a relative Rota-Baxter operator on a ternary-$F$-manifold algebra. In other words, we will give a dendrification of a ternary-$F$-manifold algebra via relative Rota-Baxter operators. 

\begin{thm}\label{thm:pre-F-algebra and F-algebra}
Let $(A,\cdot ,[\cdot,\cdot,\cdot])$ be a ternary $F$-manifold algebra and $T:V\longrightarrow
A$ be a relative Rota-Baxter operator on  $A$ with respect to the representation $(V;\rho,\mu)$. Define two operations $\diamond$ and $\{\cdot,\cdot,\cdot\}$ on $V$ as follow:
$$ u\diamond v=\mu(Tu)v,\quad 
 \{u, v,w \}=\rho(Tu,Tv)w,\ \forall \ u,v,w\in V.$$
Then $(V,\diamond,\{\cdot,\cdot,\cdot\})$ is a ternary pre-$F$-manifold algebra and $T$ is a homomorphism from $V^c$ to \\$(A,\cdot ,[\cdot,\cdot,\cdot])$. Moreover, 
\end{thm}

\begin{proof}
Since $T$ is a relative Rota-Baxter operator on the commutative associative algebra $(A,\cdot )$ as well as a relative Rota-Baxter operator on the $3$-Lie algebra $(A,[\cdot,\cdot,\cdot] )$ with respect to the representations $(V;\mu)$ and $(V;\rho)$ respectively. Using Lemma \ref{Zinbiel} and Lemma \ref{3-Lie}, we deduce that $(V,\diamond)$ is a Zinbiel algebra and $(V,\{\cdot,\cdot,\cdot\})$ is a $3$-pre-Lie algebra.
Put, for any $u,v,w \in V$, 
$$[u,v,w]_T:=\circlearrowleft_{u,v,w} \{u,v,w\}=\rho(Tu,Tv)w+\rho(Tv,Tw)u+\rho(Tw,Tu)v$$
and 
$$u\cdot_T v:=u\diamond v+v\diamond u=\mu(Tu)v+\mu(Tv)u.$$
By these facts and Eq.\eqref{eq:rep 1}, for $v_1,v_2,v_3,v_4,v_5\in V$, one has
\begin{align*}
&F_1(v_1\cdot_T v_2,v_3,v_4,v_5)-v_1\diamond F_1(v_2,v_3,v_4,v_5)-v_2\diamond F_1(v_1,v_3,v_4,v_5) \\
=& \{v_1 \cdot_T v_2,v_3,v_4\diamond v_5\}-v_4\diamond \{v_1 \cdot v_2, v_3,v_5\}-[v_1\cdot v_2,v_3,v_4]_T \diamond v_5  \\
&-v_1\diamond \{v_2,v_3,v_4\diamond v_5\}-v_1\diamond(v_4\diamond  \{v_2,v_3,v_5\})-v_1\diamond( [v_2,v_3,v_4]_T\diamond v_5)\\
&-v_2\diamond \{v_1,v_3,v_4\diamond v_5\}-v_2\diamond(v_4\diamond  \{v_1,v_3,v_5\})-v_2\diamond( [v_1,v_3,v_4]_T\diamond v_5)\\
=& \rho(T(v_1 \cdot_T v_2),Tv_3)\mu(Tv_4)v_5-\mu(Tv_4)\rho(T(v_1 \cdot v_2),T v_3)v_5-\mu(T[v_1\cdot v_2,v_3,v_4]_T) v_5  \\
&-\mu(Tv_1) \rho(Tv_2,Tv_3)\mu(Tv_4)v_5-\mu(Tv_1)\mu(Tv_4)\rho(Tv_2,Tv_3)v_5-\mu(Tv_1)\mu(T[v_2,v_3,v_4]_T) v_5\\
&-\mu(Tv_2)\rho(Tv_1,Tv_3)\mu(Tv_4) v_5-\mu(Tv_2)\mu(Tv_4)\rho(Tv_1,Tv_3)v_5-\mu(Tv_2)( \mu(T[v_1,v_3,v_4]_T)v_5\\
=&\mathcal L_1(Tv_1\cdot_T Tv_2,v_3,v_4,v_5)-\mu(Tv_1)\mathcal L_1(Tv_2,Tv_3,Tv_4,v_5)-\mu(Tv_2\mathcal L_1(Tv_1,Tv_3,Tv_4,v_5)\\
=&0,
\end{align*}
Then we obtain Eq.\eqref{eq:pre-M 1}. Similarly, using a similar computation we can obtain Eqs. \eqref{eq:pre-M 11} and \eqref{eq:pre-M 2}. Therefore,  $(V,\diamond,\{\cdot,\cdot,\cdot\})$ is a ternary pre-$F$-manifold algebra. On the other hand, for all $u,v,w\in V$ we have
$$
[Tu, Tv, Tw)] =T\Big(\rho(Tu,Tv)w+\rho(Tv,Tw)u+\rho(Tw,Tu)v\Big)=T([u,v,w]_T)
$$
and $T(u)\cdot T(v)=T(u\cdot_T v)$.
Then, $T$ is a homomorphism from $V^c$ to $(A,\cdot ,[\cdot,\cdot,\cdot])$.
\end{proof}

\begin{cor}
Let $(A,\cdot ,[\cdot,\cdot,\cdot])$ be a ternary $F$-manifold algebra and $T:V\longrightarrow
A$ an relative Rota Baxter on  $A$ with respect to the representation $(V;\rho,\mu)$. Then $T(V)=\{T(v)\mid v\in V\}\subset A$ is a subalgebra of $A$ and there is an induced ternary pre-$F$-manifold algebra structure on $T(V)$ given by
$$Tu\diamond Tv=T(u\diamond v),\quad \{Tu, Tv,Tw\}=T\{u, v,w\}, ~~~~\forall u,v,w\in V.$$
\end{cor}

\begin{cor}
 Let $(A,\cdot ,[\cdot,\cdot,\cdot])$ be a ternary $F$-manifold algebra and $\mathcal R:A\longrightarrow
A$ a Rota-Baxter operator of weight $0$. Define two operations on $A$ by
$$x\diamond y=\mathcal R(x)\cdot  y,\quad \{x, y,z\}=[\mathcal R(x),\mathcal R(y),z].$$
Then $(A,\diamond,\{\cdot,\cdot,\cdot\})$ is a  ternary pre-$F$-manifold algebra and $\mathcal R$ is a homomorphism from the sub-adjacent ternary-$F$-manifold algebra $(A,\cdot_\mathcal R,[\cdot,\cdot,\cdot]_\mathcal R)$ to $(A,\cdot ,[\cdot,\cdot,\cdot])$, where $$x\cdot_\mathcal R y=x\diamond y+y\diamond x,~~\text{and}~~ [x,y,z]_\mathcal R=\circlearrowleft_{x,y,z} \{x, y,z\}$$
for all $x,y,z\in A.$
\end{cor}

The following result is  necessary and sufficient condition on an ternary-$F$-manifold algebra to admit a
 ternary pre-$F$-manifold algebra structure.

\begin{pro}\label{pro:nsc}
 Let $(A,\cdot ,[\cdot,\cdot,\cdot])$ be an ternary $F$-manifold algebra. There is a  ternary pre-$F$-manifold algebra structure on $A$
 such that its sub-adjacent  ternary $F$-manifold algebra is exactly $(A,\cdot ,[\cdot,\cdot,\cdot])$ if and only if there exists an invertible relative Rota-Baxter operator on $(A,\cdot ,[\cdot,\cdot,\cdot])$.
\end{pro}

\begin{proof} Suppose $T:V\longrightarrow
A$ is an invertible relative Rota-Baxter operator on  $A$ with respect to the representation $(V;\rho,\mu )$, then the compatible ternary pre-$F$-manifold algebra structure on $A$ is given by
 $$x\diamond y=T(\mu(x)(T^{-1}(y))),\quad \quad \{x, y,z\}=T(\rho(x,y)(T^{-1}z))$$
for all $x,y,y\in A$. Sine $T$ is a relative Rota-baxter operator on $A$, we have 
\begin{align*} x\diamond y+y\diamond x&=T(\mu(x)(T^{-1}(y)))+T(\mu(y)(T^{-1}(x)))\\
&=T(\mu(TT^{-1}(x))(T^{-1}(y))+\mu(TT^{-1}(y))(T^{-1}(x)))\\
&=TT^{-1}(x)\cdot TT^{-1}(y)=x\cdot y
\end{align*}
and
\begin{align*} \circlearrowleft_{x,y,z}\{x,y,z\}&=T\big(\rho(x,y)(T^{-1}z)+\rho(y,z)(T^{-1}x)+\rho(z,x)(T^{-1}y)\big)\\
&=[TT^{-1}(x), TT^{-1}(y),TT^{-1}(z)]=[x, y,z].
\end{align*}

Conversely, let $(A,\diamond,\{\cdot,\cdot,\cdot\})$ be a ternary pre-$F$-manifold algebra and $(A,\cdot ,[\cdot,\cdot,\cdot])$ the sub-adjacent  ternry $F$-manifold algebra. Then the identity map $\Id$ is an relative Rota-Baxter operator on $A$ with respect to the representation $(A; \mathbb L,  L)$.
\end{proof}

Let $(A,\cdot,[\cdot,\cdot,\cdot])$ be a ternary $F$-manifold algebra and  $\mathcal B \in \wedge^2 A^*$.  We say that $\mathcal B$
is a cyclic $2$-cocycle in the sense of Connes on the commutative associative algebra $(A,\cdot)$ if 
\begin{align}
    \label{comes 2 cocycle} 
    \mathcal B(x\cdot y,z)+\mathcal B(y\cdot z,x)+\mathcal B(z \cdot x,y)=0
\end{align}
and $\mathcal B$ is said to be  a symplectic structure on the $3$-Lie algebra $(A,[\cdot,\cdot,\cdot])$ if it satisfies
\begin{align}
    \label{symplectic structure}
\mathcal B([x,y,z],u)-\mathcal B([x,y,u],z)  + \mathcal B([x, z, u], y) - \mathcal B([y, z, u], x) = 0,
\end{align}
for any $x,y,z,u \in A$.

\begin{pro}

Let $(A,\cdot ,[\cdot,\cdot,\cdot] )$ be a  coherence ternary-$F$-manifold algebra and  $\mathcal B$ be a skew-symmetric bilinear form on $A$ satisfying identities \eqref{comes 2 cocycle} and \eqref{symplectic structure}. 
Then there is a compatible ternary-pre-$F$-manifold algebra structure on $A$ given by
$$
\mathcal B(x\diamond y,z)=\mathcal B(y,x\cdot  z),\quad \mathcal B (\{x, y,z\},u)=-\mathcal B(z,[x,y,u]).
$$
\end{pro}

 \begin{proof}
   Since $(A,\cdot ,[\cdot,\cdot,\cdot] )$ is a  coherence ternary $F$-manifold algebra, then $(A^*;\ad^\ast,-\huaL^\ast)$ is a representation of $A$. By the fact that $\mathcal B$ is a symplectic structure, we deduce that the musical map $\mathcal B^\sharp:A\longrightarrow A^*$  defined by $\langle\mathcal B^\sharp(x),y\rangle=\mathcal B(x,y)$ is invertible and 
   $(\mathcal B^\sharp)^{-1}$ is a relative Rota-Baxter operator on the commutative associative algebra $(A,\cdot)$ with respect to the representation $(A^*;- L^\ast)$.
   
   In addition,  $(\mathcal B^\sharp)^{-1}$ is a  relative Rota-Baxter operator on the $3$-Lie algebra $(A,[\cdot,\cdot,\cdot])$ with respect to the representation $(A^*;\ad^\ast)$. Then, $(\mathcal B^\sharp)^{-1}$ is a relative Rota-Baxter operator on the  coherence ternary $F$-manifold algebra $(A,\cdot ,[\cdot,\cdot,\cdot])$ with respect to the representation $(A^*;\ad^\ast,- L^\ast)$. By Proposition \ref{pro:nsc}, there is a compatible  ternary pre-$F$-manifold algebra structure on $A$ given as above.
 \end{proof}

Next, let $T$ be a realtive Rota-Baxter on a ternary $F$-manifold algebra $(A,\cdot ,[\cdot,\cdot,\cdot] )$ with respect to a representation $(V;\rho,\mu)$.  According to Theorem \ref{thm:pre-F-algebra and F-algebra}, $(V,\cdot_T,[\cdot,\cdot,\cdot]_T)$ is ternary $F$-manifold algebra. Is there a representation of $V$ on the vector space $A$? The answer is affirmative. In fact, define  two linear maps $\rho_T: \wedge^2 V \to End(A)$ and $\mu_T: V \to End(A)$ by
\begin{align}\label{rep on module V1}
    \rho_T(u,v)(x)&=[Tu,Tv,x]-T(\rho(Tv,x)u+\rho(x,Tu)v),\\
\label{rep on module V2}
\mu_T(u)(x) &=Tu \cdot x-T(\mu(x)(u))
\end{align}for all $ x \in A, u,v \in V.$
\begin{thm}\label{A is rep of V}
With the above notations, the triple $(A;\rho_T,\mu_T)$ is a representation of $(V,\cdot_T,[\cdot,\cdot,\cdot]_T).$
\end{thm}

We may prove Theorem \ref{A is rep of V} by checking that $(A;\rho_T,\mu_T)$ satisfies conditions of  Definition \ref{def-rep-manifo}, but here we will prove it by another way which is more elegant. In order to do this, we should introduce and go back to some notions. Recall that a Nijenhuis operator on an associative commutative algebra $(A,\cdot)$ is a linear map $N:A \to A$ obeying to the following integrability condition:
\begin{align}
    \label{Nij op ass comm}
    Nx \cdot Ny=N(Nx \cdot y+x \cdot Ny-N(x \cdot y)), \ \forall x,y, \in A.
\end{align}
On the other hand, a linear map $N$ on  a $3$-Lie algebra $(A,[\cdot,\cdot,\cdot])$ is called a Nijenhuis operator if it satisfies, for any $x,y,z \in A$, \begin{align}
    & [Nx,Ny,Nz]=N\big([Nx,Ny,z][Nx,y,Nz]+[x,Ny,Nz] \nonumber \\
&- N[Nx,y,z]-N[x,Ny,z,]-N[x,y,Nz]+N^2[x,y,z]\big).  \label{Nij op 3-lie}
\end{align}
\begin{defi}
 Let $(A,\cdot,[\cdot,\cdot,\cdot])$ be a ternary $F$-manifold algebra. A Nijenhuis operator on $A$ is a linear map $N:A \to A$ satisfying \eqref{Nij op ass comm} and \eqref{Nij op 3-lie}. 
\end{defi}
Define two deformed products $\cdot_N$ and $[\cdot,\cdot,\cdot]_N$ on $A$  by, for all $x,y,z\in A,$
\begin{align}
    &x \cdot_N y= Nx \cdot y+x \cdot Ny-N(x \cdot y),
    \label{ass comm deformed}\\
   & [x,y,z]_N=
   [Nx,Ny,z][Nx,y,Nz]+[x,Ny,Nz] \nonumber \\
&- N[Nx,y,z]-N[x,Ny,z,]-N[x,y,Nz]+N^2[x,y,z] . \label{3-lie deormed}
\end{align}
The following two lemmas  are  straightforward, so we omit the proofs. 
\begin{lem}\label{lem-deomed by nij}
With the above notations, the tuple  $(A,\cdot_N,[\cdot,\cdot,\cdot]_N)$  is a ternary $F$-manifold algebra. 
\end{lem}

\begin{lem}\label{Nij op by T}
Let $T:V \to A$ be a relative Rota-Baxter operator on a ternary $F$-manifold algebra $(A,\cdot,[\cdot,\cdot,\cdot])$ with respect to a representation $(V;\rho,\mu)$. Then the operator $N_T=\begin{pmatrix}
  0 & T \\
  0 & 0 \\
\end{pmatrix}$ is a Nijenhuis operator on the semi-direct product ternary $F$-manifold algebra $A\oplus V$. 
\end{lem}

\emph{Proof of Theorem} 
\ref{A is rep of V}: 
According to Lemma \ref{Nij op by T}, $N_T$ is a   Nijenhuis operator on the semi-direct product ternary $F$-manifold algebra $A\oplus V$. Then, we deduce that there is a ternary $F$-manifold structure on $V \oplus A \cong A \oplus V$ given by, for all $x,y,z \in A,\ u,v,w \in V$,
\begin{align*}
    & (x+u)\cdot_{N_T} (y+v) \\
    =&N_T(x+u) \cdot_\mu (y+v) +(x+u) \cdot_\mu N_T(y+v)-N_T((x+u)\cdot_\mu (y+v)) \\
    =& Tu \cdot_\mu (y+v)+ (x+u) \cdot_\mu Tv -N_T(x\cdot y+\mu(x)v+\mu(y)u) \\
    =& Tu \cdot y+ \mu(Tu)v+  x \cdot Tv + \mu(Tv)u -
    T(\mu(x)v+\mu(y)u) \\
    =& u \cdot_T v+  \mu_T(u)y + \mu_T(v)x 
\end{align*}
and (since $N_T^2=0$), we have
\begin{align*}
& [x+u,y+v,z+w]_{N_T} \\
=&[N_T(x+u),N_T(y+v),z+w]_\rho +[N_T(x+u),y+v,N_T(z+w)]_\rho+[x+u,N_T(y+v),N_T(z+w)]_\rho \\
-&N_T( [N_T(x+u),y+v,z+w]+[x+u,N_T(y+v),z+w]+[x+u,y+v,N_T(z+w)]_\rho) \\
=& [Tu,Tv,z+w]_\rho+[Tu,y+v,Tw]_\rho+[x+u,Tv,Tw]_\rho \\
-&N_T([Tu,y+v,z+w]_\rho+[x+u,Tv,z+w]_\rho+[x+u,y+v,Tw]_\rho) \\
=& [Tu,Tv,z]+\rho(Tu,Tv)w+[Tu,y,Tw]+\rho(Tw,Tu)v+ [x,Tu,Tv]+\rho(Tv,Tw)u \\
-&T\big(\rho(Tu,y)w+\rho(z,Tu)v+\rho(x,Tv)w+\rho(Tv,z)u+\rho(y,Tw)u+\rho(Tw,x)v \big) \\
=&[u,v,w]_T+ \rho_T(u,v)z+ \rho_T(v,w)x+ \rho_T(w,u)y,
\end{align*}
which implies that $(A;\rho_T,\mu_T)$ is a representation of the ternary $F$-manifold algebra $(V,\cdot_T,[\cdot,\cdot,\cdot]_T)$. This
finishes the proof.   \qed

 \end{document}